\documentclass[10pt]{article}

\usepackage{amstext, amssymb,natbib, amsthm,booktabs,amsmath,graphicx,rotating,array}
\usepackage{algorithmic,algorithm}
\usepackage[margin=1in]{geometry}

\newtheorem{theorem}{Theorem}[section]

\newtheorem{example}[theorem]{Example}

\newtheorem{proposition}[theorem]{Proposition}
\def\tr{\mathop{\rm tr}\nolimits}
\newcommand{\ba}{\boldsymbol{a}}
\newcommand{\bb}{\boldsymbol{b}}

\newcommand{\bd}{\boldsymbol{d}}
\newcommand{\be}{\boldsymbol{e}}

\newcommand{\bu}{\boldsymbol{u}}
\newcommand{\bv}{\boldsymbol{v}}
\newcommand{\bw}{\boldsymbol{w}}
\newcommand{\bx}{\boldsymbol{x}}
\newcommand{\by}{\boldsymbol{y}}
\newcommand{\bz}{\boldsymbol{z}}
\newcommand{\bA}{\boldsymbol{A}}
\newcommand{\bB}{\boldsymbol{B}}
\newcommand{\bC}{\boldsymbol{C}}
\newcommand{\bD}{\boldsymbol{D}}
\newcommand{\bE}{\boldsymbol{E}}

\newcommand{\bI}{\boldsymbol{I}}
\newcommand{\bK}{\boldsymbol{K}}
\newcommand{\bL}{\boldsymbol{L}}
\newcommand{\bP}{\boldsymbol{P}}

\newcommand{\bM}{\boldsymbol{M}}

\newcommand{\bX}{\boldsymbol{X}}
\newcommand{\bU}{\boldsymbol{U}}
\newcommand{\bV}{\boldsymbol{V}}
\newcommand{\bW}{\boldsymbol{W}}
\newcommand{\bY}{\boldsymbol{Y}}

\newcommand{\balpha}{\boldsymbol{\alpha}}

\newcommand{\blambda}{\boldsymbol{\lambda}}

\newcommand{\bomega}{\boldsymbol{\omega}}

\title{\bf Path Following in the Exact Penalty Method \\
of Convex Programming}
\author{Hua Zhou \\
Department of Statistics \\
2311 Stinson Drive  \\ 
North Carolina State University \\
Raleigh, NC 27695-8203 \\
E-mail: hua\_zhou@ncsu.edu \\
\and
Kenneth Lange \\
Departments of Biomathematics, \\
Human Genetics, and Statistics \\
University of California \\
Los Angeles, CA 90095-1766\\
E-mail: klange@ucla.edu \\
}

\begin{document}
\maketitle

\begin{footnotetext}[1]
{Research supported in part by USPHS grants
GM53275 and MH59490 to KL, R01 HG006139 to KL and HZ, and NCSU FRPD grant to HZ.}
\end{footnotetext}

\baselineskip=20pt

\begin{abstract}
Classical penalty methods solve a sequence of unconstrained problems that put greater and greater stress on meeting the constraints.  In the limit as the penalty constant tends to $\infty$, one recovers the constrained solution.  In the exact penalty method, squared penalties are replaced by absolute value penalties, and the solution is recovered for a finite value of the penalty constant.   In practice, the kinks in the penalty and the unknown magnitude of the penalty constant prevent wide application of the exact penalty method in nonlinear programming. In this article, we examine a strategy of path following consistent with the exact penalty method. Instead of performing optimization at a single penalty constant, we trace the solution as a continuous function of the penalty constant. Thus, path following starts at the unconstrained solution and follows the solution path as the penalty constant increases. In the process, the solution path hits, slides along, and exits from the various constraints. For quadratic programming, the solution path is piecewise linear and takes large jumps from constraint to constraint. For a general convex program, the solution path is piecewise smooth, and path following operates by numerically solving an ordinary differential equation segment by segment. Our diverse applications to a) projection onto a convex set, b) nonnegative least squares, c) quadratically constrained quadratic programming, d) geometric programming, and e) semidefinite programming illustrate the mechanics and potential of path following. The final detour to image denoising demonstrates the relevance of path following to regularized estimation in inverse problems. In regularized estimation, one follows the solution path as the penalty constant decreases from a large value. 
\end{abstract}

\section{Introduction}

Penalties and barriers are both potent devices for solving constrained optimization problems
\citep{BoydVandenberghe04Book,Forsgren02interior,LuenbergerYe08Book,NocedalWright06Book,Ruszczynski06Book,Zangwill67Penalty}.
The general idea is to replace hard constraints by penalties or barriers and then exploit the well-oiled
machinery for solving unconstrained problems. Penalty methods operate on
the exterior of the feasible region and barrier methods on the interior.
The strength of a penalty or barrier is determined by a tuning constant.
In classical penalty methods, a single global tuning constant is gradually
sent to $\infty$; in barrier methods, it is gradually sent to 0. Either strategy
generates a sequence of solutions that converges in practice to the
solution of the original constrained optimization problem.

Barrier methods are now generally conceded to offer a better approach to
solving convex programs than penalty methods. Application of log barriers and
carefully controlled versions of Newton's method make it possible to follow the central path reliably and
quickly to the constrained minimum \citep{BoydVandenberghe04Book}. Nonetheless, penalty methods
should not be ruled out. Augmented Lagrangian methods \citep{Hestenes75OptmBook} and exact penalty methods
\citep{NocedalWright06Book} are potentially competitive with interior point methods for smooth
convex programming problems.   Both methods have the advantage that the solution
of the constrained problem kicks in for a finite value of the penalty constant.  This avoids
problems of ill conditioning as the penalty constant tends to $\infty$.

The disadvantage of exact penalties over traditional quadratic penalties
is lack of differentiability of the penalized objective function.   In the current paper, we argue that this impediment can be
finessed by path following.  Our path following method starts at the unconstrained
solution and follows the solution path as the penalty constant increases. In the process, the
solution path hits, exits, and slides along the various constraint boundaries.  The path itself is
piecewise smooth with kinks at the boundary hitting and escape times.  One advances along the path
by numerically solving a differential equation for the Lagrange multipliers of the penalized problem. In the special case of quadratic programming with affine constraints, the solution path is piecewise linear, and one can easily anticipate entire path segments \citep{ZhouLange11LSPath}. This special case is intimately related to the linear complementarity problem \citep{CottlePangStone92LinearComplementarity} in optimization theory.

Homotopy (continuation) methods for the solution of  nonlinear equations and optimization problems have been pursued for many years and enjoyed a variety of successes \citep{NocedalWright06Book,Watson86Homotopy,Watson01Homotopy,zangwill1981pathways}. To our knowledge, however, there has been no exploration of path following as an implementation of the exact penalty method. Our modest goal here is to assess the feasibility and versatility of exact path following for constrained optimization.
Comparing its performance to existing methods, particularly the interior point
method, is probably best left for later, more practically oriented papers.  In our experience,
coding the algorithm is straightforward in Matlab.  The rich numerical resources
of {\sc Matlab} include differential equation solvers that alert the user when certain events
such as constraint hitting and escape occur.

The rest of the paper is organized as follows. Section \ref{sec:exact-penalty-method} briefly reviews the exact penalty method for optimization and investigates sufficient conditions for uniqueness and continuity of the solution path. Section \ref{sec:path-algorithm} derives the path following strategy for general convex programs,  with particular attention to the special cases of quadratic programming and convex optimization with affine constraints. Section \ref{sec:path-exaples} presents various applications of the path algorithm. Our most elaborate example demonstrates the
relevance of path following to regularized estimation.  The particular problem treated, image denoising, is typical of many inverse problems in applied mathematics and statistics \cite{ZhouWu11EPSODE}. In such problems one follows the solution path as the penalty constant decreases. Finally, Section \ref{sec:conclusions} discusses the limitations of the path algorithm and hints at future generalizations.

\section{Exact Penalty Methods}
\label{sec:exact-penalty-method}

In this paper we consider the convex programming problem of minimizing the convex objective function $f(\bx)$ subject to $r$ affine equality constraints $g_i(\bx) = 0$ and $s$ convex inequality constraints $ h_j(\bx) \le 0$. We will further assume that $f(\bx)$ and the $h_j(\bx)$ are twice differentiable. The differential $df(\bx)$ is the row vector of partial derivatives of $f(\bx)$;  the gradient $\nabla f(\bx)$ is the transpose of $df(\bx)$. The second differential $d^2f(\bx)$ is the Hessian matrix of second partial derivatives of $f(\bx)$.  Similar conventions hold for the differentials of the constraint functions.

Exact penalty methods \citep{NocedalWright06Book,Ruszczynski06Book} minimize the surrogate function
\begin{eqnarray}
{\mathcal E}_{\rho}(\bx) & = & f(\bx)+\rho \sum_{i=1}^r |g_i(\bx)|+\rho \sum_{j=1}^s \max\{0,h_j(\bx)\}.    \label{eqn:exact-pen-obj}
\end{eqnarray}
This definition of ${\mathcal E}_{\rho}(\bx) $ is meaningful regardless of whether the contributing functions are convex.   If the program is convex, then ${\mathcal E}_{\rho}(\bx)$ is itself convex.  It is interesting to compare ${\mathcal E}_{\rho}(\bx)$ to the Lagrangian function
\begin{eqnarray*}
{\mathcal L}(\bx) & = & f(\bx)+ \sum_{i=1}^r \lambda_i g_i(\bx)+\sum_{j=1}^s \mu_j h_j(\bx) ,
\end{eqnarray*}
which captures the behavior of $f(\bx)$ near a constrained local minimum $\by$.
The Lagrangian satisfies the stationarity condition $\nabla {\mathcal L}(\by) ={\bf 0}$;
its inequality multipliers $\mu_j$ are nonnegative and satisfy the complementary slackness conditions
$\mu_j h_j(\by)=0$.  In an exact penalty method one takes
\begin{eqnarray}
\rho & > & \max \{|\lambda_1|,\ldots,|\lambda_r|,\mu_1,\ldots,\mu_s\}. \label{eqn:big-rho}
\end{eqnarray}
This choice creates the favorable circumstances
\begin{eqnarray*}
{\mathcal L}(\bx) & \le & {\mathcal E}_{\rho}(\bx) \quad  \mbox{for all} \: \bx \\
{\mathcal L}(\bz) & \le & f(\bz) \: = \: {\mathcal E}_{\rho}(\bz) \quad \mbox{for all feasible} \: \bz \\
{\mathcal L}(\by) & = & f(\by) \: = \: {\mathcal E}_{\rho}(\by) \quad \mbox{for} \; \by \; \mbox{optimal}
\end{eqnarray*}
with profound consequences.  As the next proposition proves, minimizing ${\mathcal E}_{\rho}(\bx)$
is effective in minimizing $f(\bx)$ subject to the constraints.

\begin{proposition}
\label{prop:conv-coercive}
Suppose the objective function $f(\bx)$ and the constraint functions
are twice differentiable and satisfy the Lagrange multiplier rule at the
local minimum $\by$. If inequality (\ref{eqn:big-rho}) holds and
$\bv^*d^2{\mathcal L}(\by)\bv>0$ for every vector $\bv \ne {\bf 0}$ satisfying
$dg_i(\by)\bv=0$ and $dh_j(\by)\bv \le 0$ for all active inequality constraints, then
$\by$ furnishes an unconstrained local minimum of ${\mathcal E}_{\rho}(\bx)$.
For a convex program satisfying Slater's constraint qualification and
inequality (\ref{eqn:big-rho}),  $\by$ is a minimum of ${\mathcal E}_{\rho}(\bx)$ if and only if $\by$ is a minimum
of $f(\bx)$ subject to the constraints.  No differentiability assumptions are required for convex programs.
\end{proposition}
\begin{proof}
The conditions imposed on the quadratic form $\bv^*d^2{\mathcal L}(\by)\bv$ are well-known sufficient conditions
for a local minimum.  Theorems 6.9 and 7.21 of the reference \cite{Ruszczynski06Book}
prove all of the foregoing assertions.
\end{proof}

As previously stressed, the exact penalty method turns a constrained optimization problem into an unconstrained minimization problem. Furthermore, in contrast to the quadratic penalty method \citep[Section 17.1]{NocedalWright06Book}, the constrained solution in the exact method is achieved for a finite value of $\rho$. Despite these advantages, minimizing the surrogate function ${\mathcal E}_{\rho}(\bx)$ is complicated.  For one thing, it is no longer globally differentiable.  For another, one
must minimize ${\mathcal E}_\rho(\bx)$ along an increasing sequence $\rho_n$ because the Lagrange multipliers (\ref{eqn:big-rho}) are usually unknown in advance.  These hurdles have prevented wide application of exact penalty methods in convex programming.

As a prelude to our derivation of the path following algorithm for convex programs, we record several properties of ${\mathcal E}_{\rho}(\bx)$ that mitigate the failure of differentiability.
\begin{proposition} \label{proposition1}
The surrogate function ${\mathcal E}_\rho(\bx)$ is increasing in $\rho$.  Furthermore,  ${\mathcal E}_{\rho}(\bx)$ is strictly convex for one $\rho>0$ if and only if it is strictly convex for all $\rho>0$.  Likewise, it  is coercive for one $\rho>0$ if and only if is coercive for all $\rho>0$. Finally, if $f(\bx)$ is strictly convex (or coercive), then all  ${\mathcal E}_{\rho}(\bx)$ are strictly convex (or coercive).
\end{proposition}
\begin{proof}
The first assertion is obvious. For the second assertion, consider more generally a finite family $u_1(\bx),\ldots,u_q(\bx)$ of convex functions, and suppose a linear combination $\sum_{k=1}^q c_k u_k(\bx)$ with positive coefficients is strictly convex. It suffices to prove that any other linear combination $\sum_{k=1}^q b_k u_k(\bx)$ with positive coefficients is strictly convex.  For any two points $\bx \ne \by$ and any scalar $\alpha \in (0,1)$, we have
\begin{eqnarray}
u_k[\alpha \bx+(1-\alpha) \by] \le \alpha u_k(\bx)+(1-\alpha)u_k(\by) . \label{jensen_ineq}
\end{eqnarray}
Since $\sum_{k=1}^q c_k u_k(\bx)$ is strictly convex, strict inequality must hold for at least one $k$.  Hence, multiplying inequality
(\ref{jensen_ineq}) by $b_k$ and adding gives
\begin{eqnarray*}
\sum_{k=1}^q b_k u_k[\alpha \bx+(1-\alpha) \by] & < & \alpha \sum_{k=1}^q b_k u_k(\bx)+(1-\alpha) \sum_{k=1}^q b_k u_k(\by) .
\end{eqnarray*}
The third assertion follows from the fact that a convex function is coercive if and only if its restriction to each half-line is coercive \citep[Proposition 3.2.2]{Bertsekas03ConvexBook}. Given this result, suppose ${\mathcal E}_\rho(\bx)$ is coercive, but  ${\mathcal E}_{\rho^\ast}(\bx)$ is not coercive.  Then there exists a point $\bx$, a direction $\bv$, and a sequence of scalars $t_n$ tending to $\infty$ such that ${\mathcal E}_{\rho^\ast}(\bx+t_n\bv)$ is bounded above.  This requires the sequence $f(\bx+t_n \bv)$ and each of the sequences $|g_i(\bx+t_n \bv)|$ and $\max\{0,h_j(\bx+t_n \bv)$ to remain bounded above. But in this circumstance the sequence ${\mathcal E}_\rho(\bx+t_n \bv)$ also remains bounded above. The final two assertions are obvious.
\end{proof}

\section{The Path Following Algorithm}
\label{sec:path-algorithm}

In this section, we take a different point of view. Instead  of minimizing ${\mathcal E}_{\rho}(\bx)$ for an increasing sequence $\rho_n$, we study how the solution $\bx(\rho)$ changes continuously with $\rho$ and devise a path following strategy starting from $\rho=0$. For some finite value of $\rho$,  the path locks in on the solution of the original convex program.  In regularized statistical estimation and inverse problems, the primary goal is to select relevant predictors rather than to find a constrained solution.  Thus, the entire solution path commands more interest than any single point
along it \citep{EfronHastieIainTibshirani04LARS,OsbornePresnellTurlach00LassoAlgo,TibshiraniTaylor10GenLasso,ZhouLange11LSPath,ZhouWu11EPSODE}.  Although our theory will focus on constrained estimation, readers should bear in mind this second application area of path following.

The path algorithm relies critically on the first order optimality condition that characterizes the optimum point of the convex function
${\mathcal E}_\rho(\by)$.
\begin{proposition}
For a convex program, a point $\bx=\bx(\rho)$ minimizes  the function ${\mathcal E}_\rho(\by)$  if and only if $\bx$ satisfies the
stationarity condition
\begin{eqnarray}
{\bf 0} & = & \nabla f(\bx) + \rho \sum_{i=1}^r s_i \nabla g_i(\bx) + \rho \sum_{j=1}^s t_j \nabla h_j(\bx) \label{path_stationary}
\end{eqnarray}
for coefficient sets $\{s_i\}_{i=1}^r$ and $\{t_j\}_{j=1}^s$. These sets can be characterized as
\begin{align}
   s_i \in \begin{cases}
   \{-1\} & g_i(\bx) < 0 \\
   [-1, 1] & g_i(\bx) = 0    \\
   \{1\} & g_i(\bx) > 0
   \end{cases} \hspace{.2in} \mbox{ and } \hspace{0.2in} t_j \in \begin{cases}
   \{0\} & h_j(\bx) < 0 \\
   [0, 1] & h_j(\bx) = 0    \\
   \{1\} & h_j(\bx) > 0
   \end{cases}. \label{eqn:path-subgradient}
\end{align}
At most one point achieves the minimum of ${\mathcal E}_\rho(\by)$ for a given $\rho$ when ${\mathcal E}_\rho(\by)$ is strictly convex.
\end{proposition}
\begin{proof}
According to Fermat's rule, $\bx$ minimizes ${\mathcal E}_\rho(\by)$ if and only if ${\bf 0}$ belongs to the subdifferential $\partial {\mathcal E}_\rho(\bx)$ of ${\mathcal E}_\rho(\by)$.  To derive the subdifferential displayed in equations (\ref{path_stationary}) and (\ref{eqn:path-subgradient}), one applies the addition and chain rules of the convex calculus.  The sets defining the possible values of $s_i$ and $t_j$ are the subdifferentials of the functions $|s|$ and $t_+ = \max\{t,0\}$, respectively. For more details see Theorem 3.5 and ancillary material in the book \citep{Ruszczynski06Book}. Finally, it is well known that strict convexity guarantees a unique minimum.
\end{proof}

To speak coherently of solution paths, one must validate the existence, uniqueness, and continuity of the solution $\bx(\rho)$ to the system of equations (\ref{eqn:exact-pen-obj}).  Uniqueness follows from strict convexity as already noted. Existence and continuity are more subtle.
\begin{proposition}
\label{prop:solpath-uniq-cont}
If ${\mathcal E}_\rho(\by)$ is strictly convex and coercive, then the solution path $\bx(\rho)$ of equation (\ref{eqn:exact-pen-obj}) exists and is continuous in $\rho$. If the gradient vectors $\{ \nabla g_i(\bx): g_i(\bx) = 0\} \cup \{\nabla h_j(\bx): h_j(\bx)=0\}$ of the active constraints are linearly independent at $\bx(\rho)$ for $\rho>0$, then the coefficients $s_i(\rho)$ and $t_j(\rho)$ are unique and continuous near $\rho$ as well.
\end{proposition}
\begin{proof}
In accord with Proposition \ref{proposition1}, we assume that either $f(\bx)$ is strictly convex and coercive or restrict our attention to the open interval $(0,\infty)$.  Consider a subinterval $[a,b]$ and fix a point $\bx$ in the common domain of the functions ${\mathcal E}_{\rho}(\by)$.  The coercivity of
${\mathcal E}_a(\by)$ and the inequalities
\begin{eqnarray*}
{\mathcal E}_a [\bx(\rho) ] & \le & {\mathcal E}_{\rho}[\bx(\rho)] \: \le \: {\mathcal E}_{\rho}(\bx) \: \le \: {\mathcal E}_b(\bx)
\end{eqnarray*}
demonstrate that the solution vector $\bx(\rho)$ is bounded over $[a,b]$. To prove continuity, suppose that it fails for a given $\rho \in [a,b]$.  Then there exists an $\epsilon>0$ and a sequence $\rho_n$ tending to $\rho$ such $\|\bx(\rho_n)-\bx(\rho)\|_2 \ge \epsilon$ for all $n$. Since $\bx(\rho_n)$ is bounded, we can pass to a subsequence if necessary and assume that $\bx(\rho_n)$ converges to some point $\by$.  Taking limits in the inequality ${\mathcal E}_{\rho_n}[\bx(\rho_n)] \le {\mathcal E}_{\rho_n}(\bx)$ demonstrates that ${\mathcal E}_{\rho}(\by) \le {\mathcal E}_{\rho}(\bx)$ for all $\bx$. Because $\bx(\rho)$ is unique, we reach the contradictory conclusions $\|\by-\bx(\rho)\|_2 \ge \epsilon$ and $\by = \bx(\rho)$.

Verification of the second claim is deferred to permit further discussion of path following. The claim says that an active constraint ($g_i(\bx)=0$ or $h_j(\bx)=0$) remains active until its coefficient hits an endpoint of its subdifferential. Because the solution path is, in fact, piecewise smooth, one can follow the coefficient path by numerically solving an ordinary differential equation (ODE).
\end{proof}

Our path following algorithm works segment-by-segment. Along the path we keep track of the following index sets
\begin{align}
   {\mathcal N}_{\text{E}} &= \{i: g_i(\bx) < 0\} \hspace{.5in} {\mathcal N}_{\text{I}} = \{j: h_j(\bx) < 0\}  \nonumber  \\
   {\mathcal Z}_{\text{E}} &= \{i: g_i(\bx) = 0\} \hspace{.5in} {\mathcal Z}_{\text{I}} = \{j: h_j(\bx) = 0\}  \label{eqn:set-config}  \\
   {\mathcal P}_{\text{E}} &= \{i: g_i(\bx) > 0\} \hspace{.5in} {\mathcal P}_{\text{I}} = \{j: h_j(\bx) > 0\}  \nonumber
\end{align}
determined by the signs of the constraint functions.  For the sake of simplicity, assume that at the beginning of the current segment $s_i$ does not equal $-1$ or $1$ when $i \in {\mathcal Z}_{\text{E}}$ and $t_j$ does not equal $0$ or $1$ when $j \in {\mathcal Z}_{\text{I}}$. In other words, the coefficients of the active constraints occur on the interior of their subdifferentials.  Let us show in this circumstance that the solution path can be extended in a smooth fashion.  Our plan of attack is to reparameterize by the Lagrange multipliers for the active constraints.  Thus, set
$\lambda_i = \rho s_i$ for $i \in {\mathcal Z}_{\text{E}}$ and $\omega_j = \rho t_j$ for $j \in {\mathcal Z}_{\text{I}}$.  The multipliers satisfy $-\rho < \lambda_i < \rho$ and $0 < \omega_j < \rho$. The stationarity condition now reads
\begin{eqnarray*}
{\bf 0} & = & \nabla f(\bx)-\rho \sum_{i \in {\mathcal N}_{\text{E}}} \nabla g_i(\bx)+
\rho \sum_{i \in {\mathcal P}_{\text{E}}} \nabla g_i(\bx)+\rho \sum_{j \in {\mathcal P}_{\text{I}}} \nabla h_j(\bx) \\
&   & \hspace{.5in} + \sum_{i \in {\mathcal Z}_{\text{E}}} \lambda_i \nabla g_i(\bx)+\sum_{j \in {\mathcal Z}_{\text{I}}} \omega_j \nabla h_j(\bx) .
\end{eqnarray*}
To this we concatenate the constraint equations $0=g_i(\bx)$ for $i \in {\mathcal Z}_{\text{E}}$ and $0=h_j(\bx)$ for $j \in {\mathcal Z}_{\text{I}}$.

For convenience now define
\begin{align*}
& \bU_{{\mathcal Z}}(\bx) = \left[ \!\begin{array}{c} dg_{{\mathcal Z}_{\text{E}}}(\bx) \\ dh_{{\mathcal Z}_{\text{I}}}(\bx) \end{array} \!\right],
\hspace{.15in}
   \bu_{\bar{\mathcal Z}}(\bx) = - \sum_{i \in {\mathcal N}_{\text{E}}} \nabla g_i(\bx) + \sum_{i \in {\mathcal P}_{\text{E}}} \nabla g_i(\bx) + \sum_{j \in {\mathcal P}_{\text{I}}} \nabla h_j(\bx).
\end{align*}
In this notation the stationarity equation can be recast as
\begin{eqnarray*}
{\bf 0} & = & \nabla f(\bx)+\rho \bu_{\bar{\mathcal Z}}(\bx) +  \bU_{{\mathcal Z}}^t(\bx) \left[ \begin{matrix}  \blambda \\ \bomega \end{matrix} \right].
\end{eqnarray*}
Under the assumption that the matrix $ \bU_{{\mathcal Z}}(\bx) $ has full row rank, one can solve for the Lagrange multipliers in the form
\begin{eqnarray}
\left[ \begin{matrix}  \blambda_{{\mathcal Z}_{\text{E}}} \\ \bomega_{{\mathcal Z}_{\text{I}}} \end{matrix} \right]  & = & -  [ \bU_{{\mathcal Z}}(\bx)\bU_{{\mathcal Z}}^t(\bx)]^{-1}
 \bU_{{\mathcal Z}}(\bx)\left[\nabla f(\bx)+\rho \bu_{\bar{\mathcal Z}}(\bx) \right]  . \label{short_stationaryeq}
\end{eqnarray}
Hence, the multipliers are unique.  Continuity of the multipliers is a consequence of the continuity of the solution
vector $\bx(\rho)$ and all functions in sight on the right-hand side of equation (\ref{short_stationaryeq}). This observation completes the proof of Proposition \ref{prop:solpath-uniq-cont}.

Collectively the stationarity and active constraint equations can be written as the vector equation ${\bf 0} = k(\bx,\blambda,\bomega,\rho)$.  To solve for $\bx$, $\blambda$ and $\bomega$ in terms of $\rho$, we apply the implicit function theorem \citep{Lange04Optm,MagnusNeudecker99MatrixBook}. This requires calculating the differential of $k(\bx,\blambda,\bomega,\rho)$ with respect to the underlying dependent variables $\bx$, $\blambda$, and $\bomega$ and the
independent variable $\rho$.  Because the equality constraints are affine, a brief calculation gives
\begin{eqnarray*}
\partial_{\bx,\blambda,\bomega} k(\bx,\blambda,\bomega,\rho) & = &  \left[ \begin{matrix} d^2f(\bx)+ \rho \sum_{j \in {\mathcal P}_{\text{I}}} d^2h_j(\bx)+\sum_{j \in {\mathcal Z}_{\text{I}}} \omega_j d^2h_j(\bx) & \bU_{{\mathcal Z}}^t(\bx) \\ \bU_{{\mathcal Z}}(\bx) & {\bf 0}  \end{matrix} \right]\\
\partial_\rho k(\bx,\blambda,\bomega,\rho) & = & \left( \begin{matrix}  \bu_{\bar{\mathcal Z}}(\bx) \\ {\bf 0} \end{matrix} \right).
\end{eqnarray*}
The matrix $\partial_{\bx,\blambda,\bomega} k(\bx,\blambda,\bomega,\rho)$ is nonsingular when its upper-left block is positive definite and its lower-left block has full row rank \citep[Proposition 11.3.2]{Lange10NumAnalBook}.  Given that it is nonsingular, the implicit function theorem applies, and we can in principle solve for $\bx$, $\blambda$ and $\bomega$ in terms of $\rho$.  More importantly, the implicit function theorem supplies the derivative
\begin{eqnarray}
{d \over d\rho}\! \left[ \begin{matrix} \bx \\ \blambda_{{\mathcal Z}_{\text{E}}}   \\ \bomega_{{\mathcal Z}_{\text{I}}} \end{matrix} \right]
& = & - \partial_{\bx,\blambda,\bomega} k(\bx,\blambda,\bomega,\rho)^{-1}\partial_\rho k(\bx,\blambda,\bomega,\rho),  \label{eqn:sol-ode}
\end{eqnarray}
which is the key to path following. We summarize our findings in the next proposition.
\begin{proposition}
\label{prop:ode}
Suppose the surrogate function ${\mathcal E}_\rho(\by)$ is strictly convex and coercive. If at the point $\bx(\rho_0)$ the matrix
$\partial_{\bx,\blambda,\bomega} k(\bx,\blambda,\bomega,\rho)$ is nonsingular and the coefficient of each active constraints occurs on the interior of its subdifferential, then the  solution path $\bx(\rho)$ and Lagrange multipliers $\blambda(\rho)$ and $\bomega(\rho)$ satisfy the differential equation (\ref{eqn:sol-ode}) in the vicinity of $\bx(\rho_0)$.
\end{proposition}

In practice one traces the solution path along the current time segment until either an inactive constraint becomes active or the coefficient of an active constraint hits the boundary of its subdifferential. The earliest hitting time or escape time over all constraints determines the duration of the current segment. When the hitting time for an inactive constraint occurs first, we move the constraint to the appropriate active set ${\mathcal Z}_{\text{E}}$ or ${\mathcal Z}_{\text{I}}$ and keep the other constraints in place. Similarly, when the escape time for an active constraint occurs first, we move the constraint to the appropriate inactive set and keep the other constraints in place.  In the second scenario, if $s_i$ hits the value $-1$, then we move $i$ to ${\mathcal N}_{\text{E}}$; If $s_i$ hits the value $1$, then we move $i$ to ${\mathcal P}_{\text{E}}$.  Similar comments apply when a coefficient $t_j$ hits 0 or 1. Once this move is executed, we commence path following along the new segment. Path following continues until for sufficiently large $\rho$, the sets ${\mathcal N}_{\text{E}}$, ${\mathcal P}_{\text{E}}$, and ${\mathcal P}_{\text{I}}$ are exhausted, $\bu_{\bar {\mathcal Z}} = {\bf 0}$, and the solution vector $\bx(\rho)$ stabilizes.  Our previous paper \citep{ZhouLange11LSPath} suggests remedies in the very rare situations where escape times coincide.

Path following simplifies considerably in two special cases. Consider convex quadratic programming with objective function $f(\bx) = \frac 12 \bx^t A \bx + \bb^t \bx $ and equality constraints $\bV \bx = \bd$ and inequality constraints $\bW \bx \le \be$,  where $\bA$ is positive semi-definite. The exact penalized objective function becomes
\begin{eqnarray*}
{\mathcal E}_\rho(\bx) & = &    \frac 12 \bx^t A \bx + \bb^t \bx + \rho \sum_{i=1}^s |\bv_i^t \bx - d_i| + \rho \sum_{j=1}^t (\bw_j^t \bx - e_j)_+.
\end{eqnarray*}
Since both the equality and inequality constraints are affine, their second derivatives vanish. Both $\bU_{\mathcal Z}$ and $\bu_{\bar {\mathcal Z}}$ are constant on the current path segment, and the path $\bx(\rho)$ satisfies
\begin{eqnarray}
{d \over d\rho}\! \left[ \begin{matrix} \bx \\ \blambda_{{\mathcal Z}_{\text{E}}} \\ \bomega_{{\mathcal Z}_{\text{I}}} \end{matrix} \right] & = & -
\left( \begin{matrix}
\bA & \bU_{\mathcal Z}^t \\
\bU_{\mathcal Z} & {\bf 0}
\end{matrix} \right)^{-1} \left( \begin{matrix} \bu_{\bar {\mathcal Z}} \\ {\bf 0} \end{matrix} \right). \label{eqn:sol-ode-QP}
\end{eqnarray}
This implies that the solution path $\bx(\rho)$ is piecewise linear. Our previous paper \citep{ZhouLange11LSPath} is devoted entirely to this special class of problems and highlights many statistical applications.

On the next rung on the ladder of generality are convex programs with affine constraints. For the exact surrogate
\begin{eqnarray*}
{\mathcal E}_\rho(\bx) & = &     f(\bx) + \rho \sum_{i=1}^s |\bv_i^t \bx - d_i| + \rho \sum_{j=1}^t (\bw_j^t \bx - e_j)_+,
\end{eqnarray*}
the matrix $\bU_{\mathcal Z}$ and vector $\bu_{\bar {\mathcal Z}}$ are still constant along a path segment. The relevant differential equation becomes
\begin{eqnarray}
{d \over d\rho}\! \left[ \begin{matrix} \bx \\ \blambda_{{\mathcal Z}_{\text{E}}} \\ \bomega_{{\mathcal Z}_{\text{I}}} \end{matrix} \right] & = & -
\left( \begin{matrix}
d^2f(\bx) & \bU_{\mathcal Z}^t \\
\bU_{\mathcal Z} & {\bf 0}
\end{matrix} \right)^{-1} \left( \begin{matrix} \bu_{\bar {\mathcal Z}} \\ {\bf 0} \end{matrix} \right). \label{eqn:sol-ode-conv}
\end{eqnarray}
There are two approaches for computing the right-hand side of equation (\ref{eqn:sol-ode-conv}). When $\bA = d^2f(\bx)$ is positive definite and  $\bB = \bU_{\mathcal Z}$ has full row rank, the relevant inverse amounts to
\begin{eqnarray*}
 \left( \begin{matrix} \bA & \bB^t \\ \bB & {\bf 0} \end{matrix} \right)^{-1}
& = & \left( \begin{matrix} \bA^{-1} - \bA^{-1} \bB^t [\bB \bA^{-1} \bB^t]^{-1} \bB \bA^{-1} & \bA^{-1} \bB^t [\bB \bA^{-1} \bB^t]^{-1}    \\
 [\bB \bA^{-1} \bB^t]^{-1}\bB  \bA^{-1} & -[\bB \bA^{-1} \bB^t]^{-1} \end{matrix} \right).
\end{eqnarray*}
The numerical cost of computing the inverse scales as $O(n^3) + O(|{\mathcal Z}|^3)$.  When $d^2f(\bx)$ is a constant, the inverse is computed once.  Sequentially updating it for different active sets ${\mathcal Z}$ is then conveniently organized around the sweep operator of computational statistics \citep{ZhouLange11LSPath}.  For a general convex function $f(\bx)$, every time $\bx$ changes, the inverse must be recomputed.  This burden plus the cost of computing the entries of $d^2f(\bx)$ slow the path algorithm for general convex problems.

In many applications $f(\bx)$ is convex but not necessarily strictly convex.  One can circumvent problems in inverting $d^2f(\bx)$ by
reparameterizing \citep{NocedalWright06Book}.  For the sake of simplicity, suppose that all of the constraints are affine and that $\bU_{\mathcal Z}$ has full row rank.  The set of points $\bx$ satisfying the active constraints can be written as $\bx = \bw+\bY \by$, where $\bw$ is a particular solution, $\by$ is free to vary, and the columns of $\bY \in \mathbb{R}^{n \times (n-|{\mathcal Z}|)}$ span the null space of $\bU_{\mathcal Z}$ and hence are orthogonal to the rows of $\bU_{\mathcal Z}$.  Under the null space reparameterization, $\frac{d \bx}{d\rho} = \bY \frac{d \by}{d\rho}$. Furthermore,
\begin{eqnarray*}
\bY^t d^2f(\bx)\bY  & = & d_{\by}^2 f(\bw+\bY\by) \\
\bY^t \bu_{\bar{\mathcal Z}}(\bx) & = & \nabla_{\by} \Big[-\rho \sum_{i \in {\mathcal N}_{\text{E}}} g_i(\bw+\bY\by)+
\rho \sum_{i \in {\mathcal P}_{\text{E}}}  g_i(\bw+\bY\by)+\rho \sum_{j \in {\mathcal P}_{\text{I}}}  h_j(\bw+\bY\by) \Big].
\end{eqnarray*}
It follows that equation (\ref{eqn:sol-ode-conv}) becomes
\begin{eqnarray}
\frac{d }{d\rho}\by & = & - [\bY^t d^2f(\bx) \bY]^{-1} \bY^t \bu_{\bar {\mathcal Z}} \nonumber \\
\frac{d}{d \rho} \bx & = & - \bY [\bY^t d^2f(\bx) \bY]^{-1} \bY^t \bu_{\bar {\mathcal Z}}. \label{eqn:sol-ode-alt}
\end{eqnarray}
Differentiating equation (\ref{short_stationaryeq}) gives the multiplier derivatives
\begin{eqnarray}
\frac{d}{d\rho} \left[ \begin{matrix} \blambda_{{\mathcal Z}_{\text{E}}}   \\ \bomega_{{\mathcal Z}_{\text{I}}} \end{matrix} \right] &=& -(\bU_{\mathcal Z}\bU_{\mathcal Z}^t)^{-1} \bU_{\mathcal Z} \left( d^2f(\bx) \frac{d\bx}{d\rho} + \bu_{\bar {\mathcal Z}} \right). \label{eqn:sol-ode-alt2}
\end{eqnarray}

The obvious advantage of using equation (\ref{eqn:sol-ode-alt}) is that the matrix $\bY^t d^2f(\bx) \bY$ can be nonsingular when $d^2f(\bx)$ is singular. The computational cost of evaluating the right-hand sides of equations (\ref{eqn:sol-ode-alt}) and (\ref{eqn:sol-ode-alt2}) is $O([n-|{\mathcal Z}|]^3) + O(|{\mathcal Z}|^3)$. When $n-|{\mathcal Z}|$ and $|{\mathcal Z}|$ are small compared to $n$, this is an improvement over the cost $O(n^3) + O(|{\mathcal Z}|^3)$ of computing the right-hand side of equation (\ref{eqn:sol-ode}). Balanced against this gain is the requirement of finding a basis of the null space of $\bU_{\mathcal Z}$. Fortunately, the matrix $\bY$ is constant over each path segment and in practice can be computed by taking the QR decomposition of the active constraint matrix $\bU_{\mathcal Z}$. At each kink of the solution path, either one constraint enters ${\mathcal Z}$ or one leaves. Therefore, $\bY$ can be sequentially computed by standard updating and downdating formulas \citep{LawsonHanson87LSBook,NocedalWright06Book}. Which ODE (\ref{eqn:sol-ode}) or (\ref{eqn:sol-ode-alt}) is preferable depends on the specific application. When the loss function $f(\bx)$ is not strictly convex, for example when the number of parameters exceeds the number of cases in regression, path following requires the ODE (\ref{eqn:sol-ode-alt}).  Interested readers are referred to the book \citep{NocedalWright06Book} for a more extended discussion of range-space versus null-space optimization methods.

For a general convex program, one can employ Euler's update
\begin{eqnarray*}
\left[ \begin{matrix} \bx(\rho+\Delta \rho) \\ \blambda(\rho+\Delta \rho) \\ \bomega(\rho+\Delta \rho) \end{matrix}  \right] & = &
\left[ \begin{matrix} \bx(\rho) \\ \blambda(\rho) \\ \bomega(\rho) \end{matrix} \right] + \Delta \rho
{d \over d\rho}\left[ \begin{matrix} \bx(\rho) \\ \blambda(\rho) \\ \bomega(\rho) \end{matrix} \right]
\end{eqnarray*}
to advance the solution of the ODE (\ref{eqn:sol-ode}). Euler's formula may be inaccurate for $\Delta \rho$ large.  One can correct it by fixing $\rho$ and performing one step of Newton's method to re-connect with the solution path. This amounts to replacing the position-multiplier vector by
\begin{eqnarray*}
\left[\begin{matrix} \bx \\ \blambda \\ \bomega \end{matrix} \right]-
\partial_{\bx,\blambda,\bomega} k(\bx,\blambda,\bomega,\rho)^{-1}k(\bx,\blambda,\bomega,\rho) .
\end{eqnarray*}
In practice, it is certainly easier and probably safer to rely on ODE packages such as the {\tt ODE45} function in {\sc Matlab} to advance the solution of the ODE.

\section{Examples of Path Following}
\label{sec:path-exaples}

Our examples are intended to illuminate the mechanics of path following and showcase its versatility.  As we emphasized in the introduction, we forgo comparisons with other methods.  Comparisons depend heavily on programming details and problem choices, so a premature study might well be misleading.

\begin{example}
Projection onto the Feasible Region \label{projection_example}
\end{example}

Finding a feasible point is the initial stage in many convex programs. Dykstra's algorithm \citep{Dykstra83DyAlgo,Deutsch02Book}
was designed precisely to solve the problem of projecting an exterior point onto the intersection of a finite number of closed convex sets. The projection problem also yields to our generic path following algorithm. Consider the toy example of projecting a point $\bb \in \mathbb{R}^2$ onto the intersection of the closed unit ball and the closed half space $x_1 \ge 0$ \citep{Lange04Optm}. This is equivalent to solving
\begin{eqnarray*}
    &\mbox{minimize}& \hspace{.1in} f(\bx)  =  \frac 12 \|\bx - \bb\|^2 \\
    &\mbox{subject to}& \hspace{.1in} h_1(\bx)  =  \frac 12 \|\bx\|^2 - \frac 12 \le 0, \quad  h_2(\bx)  = - x_1 \le 0.
\end{eqnarray*}
The relevant gradients and second differentials are
\begin{eqnarray*}
    \nabla f(\bx) & = &  \bx - \bb, \hspace{.2in}
    \nabla h_1(\bx) = \bx, \hspace{.2in} \nabla h_2(\bx)  =  -\left( \, \begin{matrix} 1 \\ 0 \end{matrix} \, \right)  \\
    d^2f(\bx) &=& d^2h_1(\bx) \; = \;  \bI_2, \hspace{.2in} d^2h_2(\bx)  =  {\bf 0}.
\end{eqnarray*}
Path following starts from the unconstrained solution $\bx(0)=\bb$; the direction of movement is determined by formula (\ref{eqn:sol-ode}). For $\bx \in \{\bx: \|\bx\|^2>1, x_1>0\}$, the path
\begin{eqnarray*}
    \frac{d }{d\rho}\bx &=& - [(1+\rho)\bI_2]^{-1} \bx   =  - \frac{1}{1+\rho} \bx
\end{eqnarray*}
heads toward the origin. For $\bx \in \{\bx: |x_2|>1, x_1=0\}$, the path
\begin{eqnarray*}
    \frac{d}{d\rho} \left( \begin{array}{c} \bx \\ \omega_2 \end{array} \right) &=& - \left( \begin{array}{ccc} 1+\rho & 0 & -1 \\ 0 & 1+\rho & 0 \\ -1 & 0 & 0 \end{array} \right)^{-1} \left( \begin{array}{c} x_1 \\ x_2 \\ 0 \end{array} \right)  =  - \frac{1}{1+\rho} \left( \begin{array}{c} 0 \\ x_2 \\ 0 \end{array} \right)
\end{eqnarray*}
also heads toward the origin. For $\bx \in \{\bx: \|\bx\|^2>1, x_1<0\}$, the path
\begin{eqnarray*}
    \frac{d}{d\rho} \bx &=& - [(1+\rho)\bI_2]^{-1} \left( \begin{array}{c} x_1-1 \\ x_2 \end{array} \right)  =  - \frac{1}{1+\rho} \left( \begin{array}{c} x_1-1 \\ x_2 \end{array} \right).
\end{eqnarray*}
heads toward the point $(1,0)^t$.
For $\bx \in \{\bx: \|\bx\|^2=1, x_1<0\}$, the path
\begin{eqnarray*}
    \frac{d}{d\rho} \left( \begin{array}{c} \bx \\ \omega_1 \end{array} \right) &=& - \left( \begin{array}{ccc} 1+\omega_1 & 0 & x_1 \\ 0 & 1+\omega_1 & x_2 \\ x_1 & x_2 & 0 \end{array} \right)^{-1} \left( \begin{array}{c} -1 \\ 0 \\ 0 \end{array} \right)
  =  \left( \begin{array}{c} - \frac{x_2^2}{1+\omega_1} \\ \frac{x_1x_2}{1+\omega_1} \\ - x_1 \end{array} \right)
\end{eqnarray*}
is tangent to the circle. Finally, for $\bx \in \{\bx: \|\bx\|^2<1, x_1<0\}$, the path
\begin{eqnarray*}
    \frac{d}{d\rho}\bx &=& - \bI_2^{-1} \left( \begin{array}{c} -1 \\ 0 \end{array} \right)  =  \left( \begin{array}{c} 1 \\ 0 \end{array} \right)
\end{eqnarray*}
heads toward the $x_2$-axis. The left panel of Figure \ref{fig:halfdisc} plots the vector field $\frac{d}{d\rho} \bx$ at the time $\rho=0$. The right panel shows the solution path for projection from the points $(-2,0.5)^t$, $(-2,1.5)^t$, $(-1,2)^t$, $(2,1.5)^t$, $(2,0)^t$,
$(1,2)^t$, and $(-0.5,-2)^t$ onto the feasible region. In projecting the point $\bb =(-1,2)^t$ onto $(0,1)^t$, the {\sc ODE45} solver of {\sc Matlab} evaluates derivatives at 19 different time points. Dykstra's algorithm by comparison takes about 30 iterations to converge \citep{Lange04Optm}.

\begin{figure}
\begin{center}
$$
\begin{array}{cc}
\includegraphics[width=2.25in]{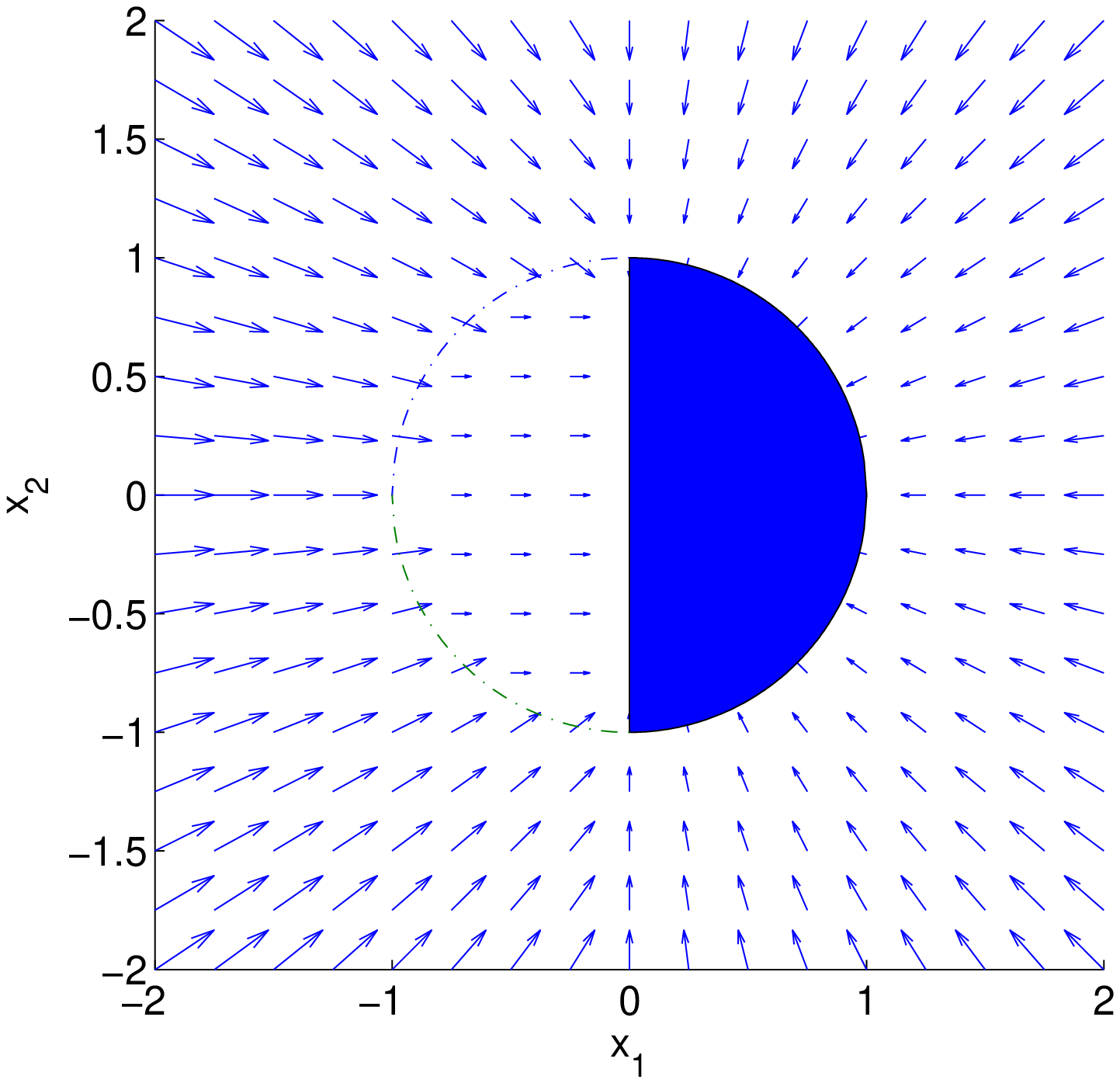} & \includegraphics[width=2.25in]{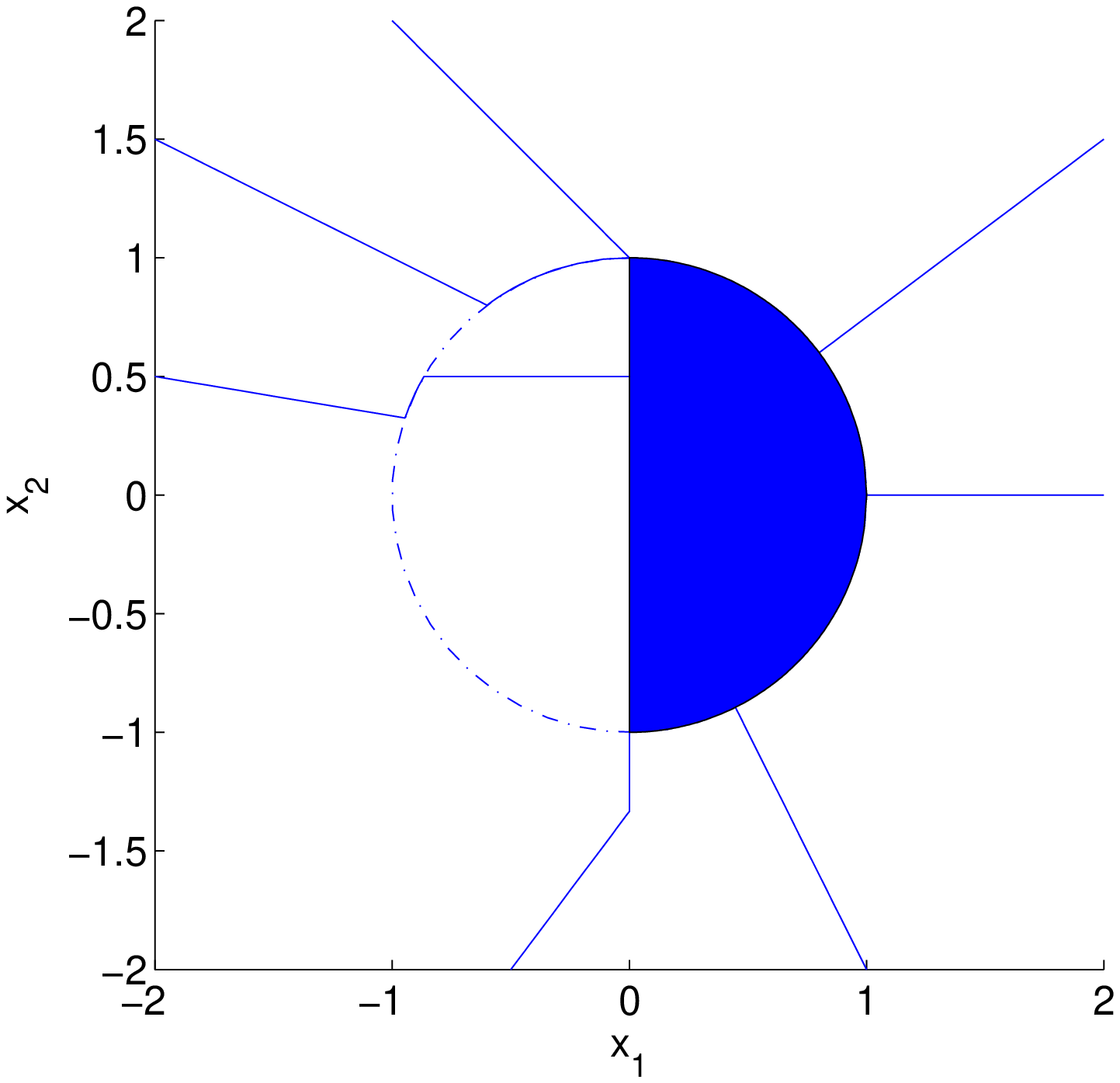}
\end{array}
$$
\end{center}
\caption{Projection to the positive half disk. Left: Derivatives at $\rho=0$ for projection onto the half disc. Right: Projection trajectories from various initial points.}
\label{fig:halfdisc}
\end{figure}

\begin{example}
Nonnegative Least Squares (NNLS) and Nonnegative Matrix Factorization (NNMF)
\end{example}

Non-negative matrix factorization (NNMF) is an alternative to principle component analysis and is useful in modeling, compressing, and interpreting nonnegative data such as observational counts and images. The articles \citep{Berry07NNMF,LeeSeung99NNMF,LeeSeung01NNMFAlgo} discuss in detail estimation algorithms and statistical applications of NNMF. The basic idea is to approximate an $m \times n$ data matrix $\bX=(x_{ij})$ with nonnegative entries by a product $\bV \bW$ of two low rank matrices $\bV=(v_{ik})$ and $\bW=(w_{kj})$ with nonnegative entries. Here $\bV$ and $\bW$ are $m \times r$ and $r \times n$ respectively, with $r \ll \min\{m,n\}$. One version of NNMF minimizes the criterion
\begin{eqnarray}
f(\bV,\bW) & = & \| \bX - \bV \bW \|_{\text{F}}^2  =
\sum_i \sum_j \Big( x_{ij} - \sum_k v_{ik} w_{kj} \Big)^2,  \label{eqn:nnmf-objfn}
\end{eqnarray}
where $\|\cdot\|_{\text{F}}$ denotes the Frobenius norm. In a typical imaging problem, $m$ (number of images) might range from $10^3$ to $10^4$, $n$ (number of pixels per image) might surpass $10^4$, and a rank $r =50$ approximation might adequately capture $\bX$.

Minimization of the objective function~(\ref{eqn:nnmf-objfn}) is nontrivial because it is not jointly convex in $\bV$ and $\bW$. Multiple local minima are possible.  The well-known multiplicative algorithm \citep{LeeSeung99NNMF,LeeSeung01NNMFAlgo} enjoys the descent property, but it is not guaranteed to converge to even a local minimum \citep{Berry07NNMF}. An alternative algorithm that exhibits better convergence is  alternating least squares (ALS). In updating $\bW$ with $\bV$ fixed, ALS solves the $n$ separated nonnegative least square (NLS) problems
\begin{eqnarray}
	\min_{\bw_j} \|\bx_j - \bV\bw_j\|_2^2 \quad \quad \text{ subject to } \bw_j \ge 0, \label{eqn:NNMF-ALS}
\end{eqnarray}
where $\bx_j$ and $\bw_j$ denote the $j$-th columns of the corresponding matrices. Similarly, in updating $\bV$ with $\bW$ fixed, ALS solves $m$ separated NNLS problems. The unconstrained solution $\bW(0) = (\bV^t \bV)^{-1} \bV^t \bX$ of $\bW$ for fixed $\bV$ requires just one QR decomposition of $\bV$ or one Cholesky decomposition of $\bV^t \bV$. The exact path algorithm for solving the subproblem problem (\ref{eqn:NNMF-ALS}) commences with $\bW(0)$.  If $\bW(\rho)$ stabilizes with just a few zeros, then the path algorithm ends quickly and is extremely efficient. For a NNLS problem, the path is piecewise linear, and one can straightforwardly project the path to the next hitting or escape time using the sweep operator \citep{ZhouLange11LSPath}. Figure \ref{fig:nnls} shows a typical piecewise linear path for a problem with $r=50$ predictors. Each projection to the next event requires $2r^2$ flops. The number of path segments (events) roughly scales as the number of negative components in the unconstrained solution.

\begin{figure}
\begin{center}
\includegraphics[width=3in]{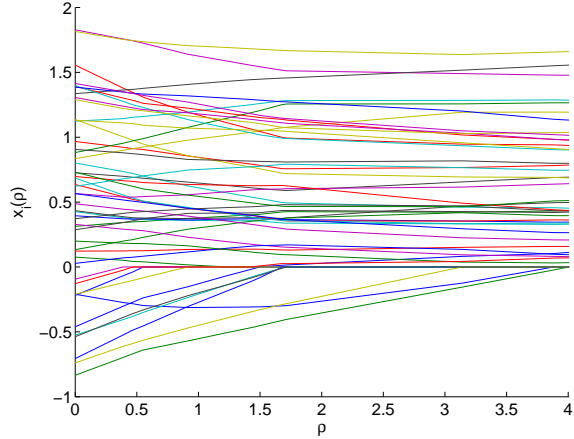}
\end{center}
\caption{Piecewise linear paths of the regression coefficients for a NNLS problem with 50 predictors.}
\label{fig:nnls}
\end{figure}

\begin{example}
Quadratically Constrained Quadratic Programming (QCQP)
\end{example}

Example \ref{projection_example} is a special case of quadratically constrained quadratic programming (QCQP). In convex QCQP \citep[Section 4.4]{BoydVandenberghe04Book}, one minimizes a convex quadratic function over an intersection of ellipsoids and affine subspaces. Mathematically,
this amounts to the problem
\begin{eqnarray*}
    \text{minimize} \hspace{.1in} f(\bx)  &=&  \frac 12 \bx^t \bP_0 \bx + \bb_0^t \bx + c_0 \\
    \text{subject to} \hspace{.1in} g_i(\bx)  &=&  \ba_i^t \bx - d_i = 0, \hspace{.1in} i=1,\ldots,r \\
    h_j(\bx)  &=&  \frac 12 \bx^t \bP_j \bx + \bb_j^t \bx + c_j \le 0, \hspace{.1in} j=1,\ldots,s,
\end{eqnarray*}
where $\bP_0$ is a positive definite matrix and the $\bP_j$ are positive semidefinite matrices. Our algorithm starts with the unconstrained minimum $\bx(0) = - \bP_0^{-1} \bb_0$ and proceeds along the path determined by the derivative
\begin{eqnarray*}
    \frac{d}{d\rho} \left[ \begin{matrix} \bx \\ \blambda_{{\mathcal Z}_{\text{E}}} \\ \bomega_{{\mathcal Z}_{\text{I}}} \end{matrix} \right] = - \left( \begin{matrix} \bP_0+ \rho \sum_{j \in {\mathcal P}_{\text{I}}} \bP_j+\sum_{j \in {\mathcal Z}_{\text{I}}} \omega_j \bP_j & \bU_{{\mathcal Z}}^t(\bx) \\ \bU_{{\mathcal Z}}(\bx) & {\bf 0}  \end{matrix} \right)^{-1} \left( \begin{matrix}  \bu_{\bar{\mathcal Z}}(\bx) \\ {\bf 0} \end{matrix} \right),
\end{eqnarray*}
where $\bU_{\mathcal Z}(\bx)$ has rows $\ba_i^t$ for $i \in {\mathcal Z}_{\text{E}}$ and $(\bP_j\bx+\bb_j)^t$ for $j \in {\mathcal Z}_{\text{I}}$, and
\begin{eqnarray*}
    \bu_{\bar {\mathcal Z}}(\bx) = - \sum_{i \in {\mathcal N}_{\text{E}}} \ba_i + \sum_{i \in {\mathcal P}_{\text{E}}} \ba_i + \sum_{i \in {\mathcal P}_{\text{I}}} (\bP_j \bx + \bb_j).
\end{eqnarray*}
Affine inequality constraints can be accommodated by setting one or more of the $\bP_j$ equal to ${\bf 0}$.

As a numerical illustration, consider the bivariate problem
\begin{eqnarray}
    \text{minimize} \hspace{.1in} f(\bx)  &=&  \frac 12 x_1^2 + x_2^2 - x_1x_2 + \frac 12 x_1 - 2x_2  \nonumber \\
    \text{subject to} \hspace{.1in} h_1(\bx)  &=&  \Big(x_1-\frac 12 \Big)^2 + x_2^2 - 1 \; \le \; 0 \; \label{eqn:QCQP-toy}  \\
    h_2(\bx)  &=&  \Big(x_1+\frac 12 \Big)^2 + x_2^2 - 1 \le  0   \nonumber \\
    h_3(\bx) &=&  x_1^2 + \Big(x_2-\frac 12 \Big)^2 - 1  \le  0. \nonumber
\end{eqnarray}
Here the feasible region is given by the intersection of three disks with centers $(0.5,0)^t$,  $(-0.5,0)^t$, and $(0,0.5)^t$, respectively, and a common radius of 1. Figure \ref{fig:3disk} displays the solution trajectory. Starting from the unconstrained minimum $\bx(0)= (1,1.5)^t$, it hits, slides along, and exits two circles before its journey ends at the constrained minimum $(0.059,0.829)^t$. The {\tt ODE45} solver of {\sc Matlab} evaluates derivatives at 72 time points along the path.

\begin{example}
Geometric Programming
\end{example}

\begin{figure}
\begin{center}
$$
\includegraphics[width=3in]{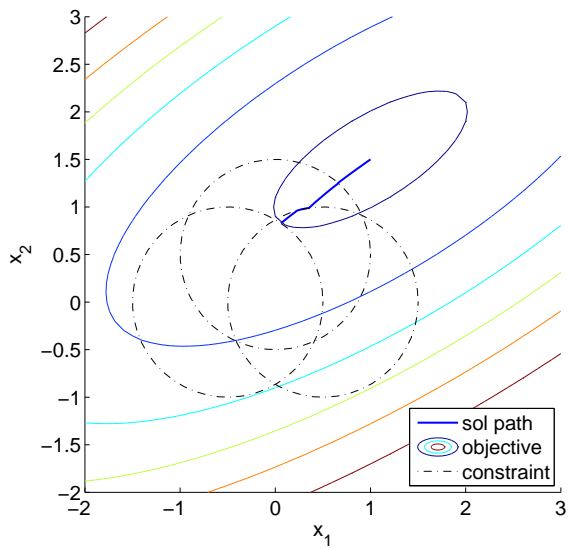}
$$
\end{center}
\caption{Trajectory of the exact penalty path algorithm for a QCQP problem (\ref{eqn:QCQP-toy}). The solid lines are the contours of the objective function $f(\bx)$. The dashed lines are the contours of the constraint functions $h_j(\bx)$.}
\label{fig:3disk}
\end{figure}

As a branch of convex optimization theory, geometric programming stands just behind linear and quadratic programming in importance \citep{Boyd07GP,Ecker80GP,Peressini88Book,Peterson76GP}. It has applications in chemical equilibrium problems \citep{PassyWilde68GP}, structural mechanics \citep{Ecker80GP}, digit circuit design \citep{Boyd05digitalcircuit}, maximum likelihood estimation \citep{MazumdarJefferson83GP}, stochastic processes \citep{FeiginUry81GP}, and a host of other subjects \citep{Boyd07GP,Ecker80GP}. Geometric programming deals with posynomials, which are functions of the form
\begin{eqnarray}
f(\bx) & = & \sum_{\balpha \in S} c_{\balpha} \prod_{i=1}^n x_i^{\alpha_{i}} \;\; = \;\; \sum_{\balpha \in S} c_{\balpha} e^{\balpha^t \by}
\: = \: f(\by).
\label{general_posynomial}
\end{eqnarray}
In the left-hand definition of this equivalent pair of definitions, the index set $S \subset \mathbb{R}^n$ is finite, and all coefficients $c_{\balpha}$ and all components $x_1,\ldots,x_n$ of the argument $\bx$ of $f(\bx)$ are positive.  The possibly fractional powers $\alpha_i$ corresponding to a particular $\balpha$ may be positive, negative, or zero.  For instance, $x_1^{-1}+2x_1^3x_2^{-2}$ is a posynomial on $\mathbb{R}^2$.  In geometric programming, one minimizes a posynomial $f(\bx)$ subject to posynomial inequality constraints of the form $h_j(\bx) \le 1$ for $1 \le j \le s$. In some versions of geometric programming, equality constraints of monomial type are permitted \citep{Boyd07GP}.  The right-hand definition in equation
(\ref{general_posynomial}) invokes  the exponential reparameterization  $x_i = e^{y_i}$.  This simple transformation has the advantage of rendering a geometric program convex.  In fact, any posynomial  $f(\by)$ in the exponential parameterization is log-convex and therefore convex.  The concise representations
\begin{eqnarray*}
\nabla f(\by) & = & \sum_{\balpha \in S} c_{\balpha} e^{\balpha^t \by} \balpha, \quad
d^2f(\by) \;\; = \;\; \sum_{\balpha \in S} c_{\balpha} e^{\balpha^t \by} \balpha \balpha^t
\end{eqnarray*}
of the gradient and the second differential are helpful in both theory and computation.

Without loss of generality, one can repose geometric programming as
\begin{eqnarray}
    \text{minimize} \hspace{.1in} \ln f(\by) & &  \nonumber \\
    \text{subject to} \hspace{.1in} \ln g_i(\by)  &=&  0, \;\; 1 \le i \le r \label{eqn:GP-convexform}  \\
    \ln h_j(\by)  &\le&  0, \;\; 1 \le j \le s,  \nonumber
\end{eqnarray}
where $f(\by)$ and the $h_j(\by)$ are posynomials and the equality constraints $\ln g_i(\by)$ are affine. In this exponential parameterization setting, it is easy to state necessary and sufficient conditions for strict convexity and coerciveness.
\begin{proposition}
\label{lem:QP}
The objective function $f(\by)$ in the geometric program (\ref{eqn:GP-convexform}) is strictly convex if and only if the subspace spanned by the vectors $\{\balpha\}_{\balpha \in S}$ is all of $\mathbb{R}^n$; $f(\by)$ is coercive if and only if the polar cone
$\{ \bz : \bz^t \balpha \le 0 \; \mbox{for all} \;  \balpha \in S\}$ reduces to the origin ${\bf 0}$. Equivalently, $f(\by)$ is coercive if the origin ${\bf 0}$ belongs to the interior of the convex hull of the set $S$.
\end{proposition}
\begin{proof}
These claims are proved in detail in our paper \citep{ZhouLange11GP}.
\end{proof}

According to Propositions \ref{prop:conv-coercive} and \ref{prop:solpath-uniq-cont}, the strict convexity and coerciveness of $f(\by)$ guarantee the uniqueness and continuity of the solution path in $\by$. This in turn implies the uniqueness and continuity of the solution path in the original parameter vector $\bx$.  The path directions are related by the chain rule
\begin{eqnarray*}
    \frac{d}{d\rho} x_i(\rho) &  = & \frac{d x_i}{dy_i} \frac{d}{d\rho} y_i(\rho) \;\; = \;\;  x_i \frac{d}{d\rho} y_i(\rho).
\end{eqnarray*}

As a concrete example, consider the problem
\begin{eqnarray}
    &\text{minimize}& \hspace{.1in} x_1^{-3} + 3 x_1^{-1} x_2^{-2} + x_1 x_2 \label{eqn:GP-toy}    \\
    &\text{subject to}& \hspace{.1in} \frac 16 x_1^{1/2} + \frac 23 x_2  \le  1, \;\; x_1>0, \; x_2>0. \nonumber
\end{eqnarray}
It is easy to check that the vectors $\{(-3,0)^t,(-1,-2)^t,(1,1)^t\}$ span $\mathbb{R}^2$ and generate a convex hull strictly containing the origin ${\bf 0}$. Therefore, $f(\by)$ is strictly convex and coercive.  It achieves its unconstrained minimum at the point $\bx(0) = (\sqrt[5]{6}, \sqrt[5]{6})^t$, or equivalently $\by(0) = (\ln 6/5,\ln 6/5)^t$. To solve the constrained minimization problem, we follow the path dictated by the revised geometric
program (\ref{eqn:GP-convexform}).
Figure \ref{fig:gp-toy1} plots the trajectory from the unconstrained solution to the constrained solution in the original $\bx$ variables. The solid lines in the figure represent the contours of the objective function $f(\bx)$, and the dashed lines represent the contours of the constraint function $h(\bx)$. The {\tt ODE45} solver of {\sc Matlab} evaluates derivatives at seven time points along the path.

\begin{figure}
\begin{center}
$$
\includegraphics[width=3in]{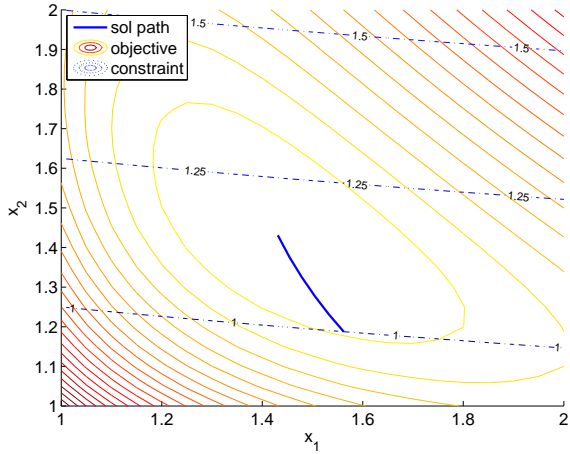}
$$
\end{center}
\caption{Trajectory of the exact penalty path algorithm for the geometric programming problem (\ref{eqn:GP-toy}). The solid lines are the contours of the objective function $f(\bx)$. The dashed lines are the contours of the constraint function $h(\bx)$ at levels 1, 1.25, and 1.5.}
\label{fig:gp-toy1}
\end{figure}

\begin{example}
Semidefinite Programming (SDP)
\end{example}

The linear semidefinite programming problem \citep{VandenbergheBoyd96SDP} consists in minimizing the trace
function $\bX \mapsto \tr(\bC \bX)$ over the cone of positive semidefinite matrices
$S_+^n$ subject to the linear constraints $\tr(\bA_i \bX) = b_i$ for $1 \le i \le p$.
Here $\bC$ and the $\bA_i$ are assumed symmetric. According to Sylvester's criterion,
the constraint $\bX \in S_+^n$ involves a complicated system of inequalities
involving nonconvex functions.  One way of cutting through this morass is to focus
on the minimum eigenvalue $\nu_1(\bX)$ of $\bX$. Because the function $-\nu_1(\bX)$
is convex, one can enforce positive semidefiniteness by requiring $-\nu_1(\bX) \le 0$.
Thus, the linear semidefinite programming problem is a convex program in the
standard functional form.

It simplifies matters enormously to assume that $\nu_1(\bX)$ has multiplicity 1.
Let $\bu$ be the unique, up to sign, unit eigenvector corresponding to $\nu_1(\bX)$.   The matrix
$\bX$ is parameterized by the entries of its lower triangle. With these conventions,
the following formulas
\begin{eqnarray}
-\frac{\partial}{\partial x_{ij}}  \nu_1(\bX) & = &  -\bu^t \frac{\partial}{\partial x_{ij}} \bX \bu \label{eqn:v1-grad} \\
-\frac{\partial^2 }{\partial x_{ij} \partial x_{kl} }  \nu_1(\bX)
& = & - \bu^t  \frac{\partial}{\partial x_{ij}} \bX (\nu_1 \bI-\bX)^- \frac{\partial}{\partial x_{kl}}  \bX \bu \nonumber \\
&   & - \bu^t  \frac{\partial}{\partial x_{kl}} \bX (\nu_1 \bI-\bX)^- \frac{\partial}{\partial x_{ij}}  \bX \bu \nonumber \\
& = & - 2 \bu^t  \frac{\partial}{\partial x_{ij}} \bX (\nu_1 \bI-\bX)^- \frac{\partial}{\partial x_{kl}}  \bX \bu \label{eqn:v1-hess}
\end{eqnarray}
for the first and second partial derivatives of $-\nu_1(\bX)$ are well known \citep{MagnusNeudecker99MatrixBook}.
Here the matrix $(\nu_1 \bI - \bX)^-$ is the Moore-Penrose inverse of $\nu_1 \bI - \bX$.
The partial derivative of $\bX$ with respect to its lower triangular entry $x_{ij}$
equals $\bE_{ij}+1_{\{i \ne j\}} \bE_{ji}$, where $\bE_{ij}$ is the matrix consisting of all 0's excepts for a 1
in position $(i,j)$. Note that $\bu^t \bE_{ij} = u_i \be_j^t$ and $\bE_{kl} \bu = u_l \be_k$ for the
standard unit vectors $\be_j$ and $\be_k$.  The second partial derivatives of $\bX$ vanish.
The Moore-Penrose inverse is most easily expressed in terms of the spectral decomposition
of $\bX$.  If we denote the $i$th eigenvalue of $\bX$ by $\nu_i$ and the corresponding $i$th unit eigenvector
by $\bu_i$, then we have
\begin{eqnarray*}
(\bX-\nu_1 \bI)^- & = & \sum_{i > 1} \frac{1}{\nu_i-\nu_1} \bu_i \bu_i^t.
\end{eqnarray*}
Finally, the formulas
\begin{eqnarray*}
\tr(\bA_i \bX) -b_i  & = & \sum_k (\bA_i)_{kk}x_{kk} + 2\sum_k \sum_{l<k} (\bA_i)_{kl}x_{kl} -b_i \\
\frac{\partial}{\partial x_{kl}} [\tr(\bA_i \bX)-b_i] & = & (\bA_i)_{kl}+1_{\{k \ne l\}} (\bA_i)_{lk}
\end{eqnarray*}
express the linear constraints and their partial derivatives in terms of the lower
triangular entries of $\bX$.

Initiating path following is problematic because $\tr(\bC \bX)$ has minimum $-\infty$.
A good strategy is to amend the surrogate function ${\mathcal E}_{\rho}(\bx)$ by adding the term
$\frac{\epsilon(\rho)}{2} \|\bX\|_{\text{F}}^2$, where $\epsilon(\rho)$ is a smooth positive function that
decreases to 0. Taking $\epsilon(\rho) = e^{-c\rho}$ for $c$ positive works well in practice.
The new surrogate function $\tr(\bC \bX) + \frac{\epsilon(\rho)}{2} \|\bX\|_{\text{F}}^2$ is strictly convex
and possesses a unique minimum for all $\rho \ge 0$. In view of the identities
$ \|\bX\|_{\text{F}}^2 = \sum_i \sum_j x_{ij}^2 $ and $\tr(\bC \bX) = \sum_i \sum_j c_{ij} x_{ij}$
for $\bX=(x_{ij})$ and $\bC = (c_{ij})$, the initial condition $\bX(0) = - \epsilon(0)^{-1} \bC$
is straightforward to deduce.

Path following must be modified to accommodate the new surrogate function. In the notation of \citep{MagnusNeudecker99MatrixBook}, let $\bx = \text{v}(\bX)$ be the $\frac 12 n(n+1)$ vector obtained from $\text{vec}(\bX)$ by eliminating all supradiagonal entries, and let $\bD$ be the $n^2 \times \frac 12 n(n+1)$ duplication matrix satisfying $\text{vec}(\bX) = \bD \bx$.  Applying the chain rule to the obvious identities $ \|\bX\|_{\text{F}}^2 = \bx \bD^t \bD \bx$
and $\tr(\bC \bX) = \text{vec}(\bC)^t \bD \bx$, one can extend the derivation of Proposition \ref{prop:ode} and prove that
\begin{eqnarray*}
&    & {d \over d\rho}\! \left[ \begin{matrix} \bx \\ \blambda_{{\mathcal Z}_{\text{E}}}   \\ \bomega_{{\mathcal Z}_{{\mathcal I}}} \end{matrix} \right] \\
& = & - \left[ \begin{matrix} \epsilon(\rho) \bD^t \bD -  \bomega_{{\mathcal Z}_{{\mathcal I}}} d^2 \nu_1(\bx) 1_{\{\nu_1(\bX)=0\}} & \bU_{{\mathcal Z}}^t \\
\bU_{{\mathcal Z}} & {\bf 0}
\end{matrix} \right]^{-1} \\
& & \! \times \! \left( \begin{matrix}
    \frac{d\epsilon(\rho)}{d \rho} \bD^t \bD \bx  -  \sum_{i \in {\mathcal N}_{\text{E}}} \bD^t \text{vec}(\bA_i) +  \sum_{i \in {\mathcal P}_{\text{E}}} \bD^t \text{vec}(\bA_i)  -  \nabla \nu_1(\bx)  1_{\{\nu_1(\bX)<0\}}  \\
    {\bf 0}
\end{matrix} \right).
\end{eqnarray*}
Path following proceeds until all constraints are satisfied and $\epsilon(\rho)$ is negligible.

For didactic purposes, considering the problem of minimizing $\tr(\bC \bX)$ subject to
\begin{eqnarray*}
    \tr(\bA_1\bX) & = & 1, \quad \tr(\bA_2\bX)=2, \:\: \text{ and } \:\: \:  \bX \in {\mathcal S}_+^2 ,
\end{eqnarray*}
where
\begin{eqnarray*}
    \bC = \left( \begin{matrix} 0 & \frac 12 \\ \frac 12 & 0 \end{matrix} \right), \quad \bA_1  =  \left( \begin{matrix} 1 & 0 \\ 0 & 0 \end{matrix} \right), \:\: \text{ and } \:\:\: \bA_2 = \left( \begin{matrix} 0 & 0 \\ 0 & 1 \end{matrix} \right).
\end{eqnarray*}
Figure \ref{fig:sdp-toy1} displays the solution paths of the entries $x_{ij}$ of $\bX$ and the minimum eigenvalue $\nu_1$ . Here we use $\epsilon(\rho)=e^{-\rho}$. The path starts with $\bX(0) = - \bC$, hits, slides along, and exits various constraints, and ends at the constrained solution $
\small \left(\begin{matrix} 1 & -\sqrt{2} \\ -\sqrt{2} & 2 \end{matrix} \right)$.

\begin{figure}
\begin{center}
\includegraphics[width=3in]{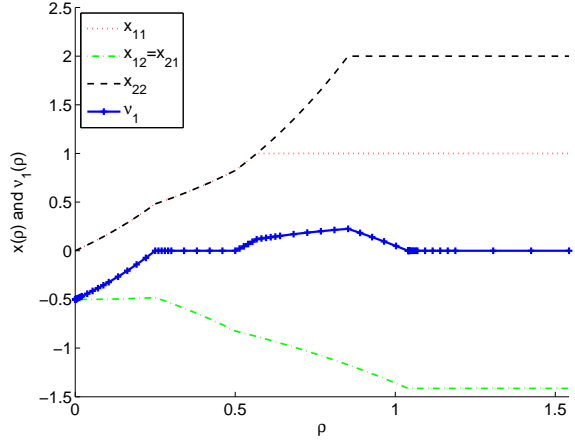}
\end{center}
\caption{Solution path of a semidefinite programming example.}
\label{fig:sdp-toy1}
\end{figure}

\begin{example}
Image Denoising
\end{example}

Image analysis is another fertile field for path following. Here we explore how to restore or enhance images by removing noise. This example differs from previous examples in that the fully constrained solution is trivial. The solution path itself is the object of interest.  Suppose that ${\bw} = (w_{ij}) \in \mathbb{R}^{m \times n}$ represents the recorded gray levels across a 2D array of pixels from a noisy image with true gray levels $\bu = (u_{ij})$. The well-known denoising model of Rudin-Osher-Fatemi (ROF) \citep{Rudin92ROF} minimizes the total variation regularized least squares criterion
\begin{eqnarray}
 & & \frac 12 \| {\bw}-\bu\|_2^2 + \rho \text{TV}(\bu) \nonumber \\
    &=& \frac 12 \sum_{ij} (w_{ij}-u_{ij})^2 + \rho \sum_{i,j} \sqrt{(u_{i+1,j}-u_{ij})^2 + (u_{i,j+1}-u_{ij})^2}. \label{osher_problem}
\end{eqnarray}
The total variation penalty serves to smooth the reconstructed image and preserve its edges. A similar effect can be achieved by replacing the isotropic penalty $\text{TV}(\bu)$ by the anisotropic penalty
\begin{eqnarray}
 \text{TV}_{1}(\bu) & = & \sum_{i,j} \Big(|u_{i+1,j}-u_{ij}|+|u_{i,j+1}-u_{ij}|\Big). \label{TV1_penalty}
\end{eqnarray}

In this example we focus on path following for the anisotropic penalty and a more general convex loss function $f(\bu)$. The objective function is now
\begin{eqnarray}
f(\bu) + \rho \|\bD \bu\|_1. \label{eqn:denoise-gen}
\end{eqnarray}
For instance, the amended loss function $f(\bu)=\frac{1}{2} \|\bw - \bK \bu \|_2^2$ with a Gaussian or motion blurring matrix $\bK$ is appropriate in many imaging problems. Poisson count data are relevant to image reconstruction in X-ray and positron tomography \citep{Lange10NumAnalBook} and to image denoising in certain circumstances \citep{LeChartrandAsaki07PoissonROF}.  With Poisson noise,  the least squares criterion is replaced by a negative loglikelihood. The difference matrix $\bD$ captures the $\ell_1$ penalty (\ref{TV1_penalty}). Note that the matrices $\bw$ and $\bu$ are now viewed as vectors. For an $m \times n$ 2D image, the difference matrix $\bD$ has $2mn-m-n$ rows  (penalties) and $mn$ columns (pixels). This matrix is very sparse, with just $2(2mn-m-n)$ nonzero entries equal to $\pm1$. When $m$ and $n$ are both at least 2, $\bD$ has more rows than columns and a reduced column rank of $mn-1$.

For sufficiently large $\rho$, the minimum of the objective functions (\ref{osher_problem}) reduces to a constant vector (blank image) equal to the average value $\bar{w}$ of the $w_{ij}$. The goal of image denoising is to find a $\rho$ such that the recovered image is judged satisfactory by visual inspection or other more quantitative criteria.  Notable computational advances in solving this problem include Chambolle's algorithm \citep{Chambolle04TVAlgo} and split Bregman iteration \citep{Goldstein09SplitBregman}.  These methods minimize the objective functions (\ref{osher_problem}) and (\ref{eqn:denoise-gen}) for a fixed value of $\rho$. The web site of UCLA's Computational and Applied Math Group summarizes the most recent progress in this area. In reality, outer iterations are almost always required to tune the parameter $\rho$. Path following is an attractive option because it provides the whole solution path at about the same computational cost as recovering the solution for an individual $\rho$.

Although it is tempting to minimize the criterion (\ref{eqn:denoise-gen}) by path following, the regularization matrix $\bD$ has linearly dependent rows and deficient rank.  Because the assumptions of Proposition \ref{prop:solpath-uniq-cont} are violated, the multipliers $\blambda_{\text{E}}$ of the active constraints in equations (\ref{short_stationaryeq}) and (\ref{eqn:sol-ode-QP}) are not uniquely determined. One can intuitively understand the difficulty by considering a square with four pixels. Whenever any three constraints are active, the fourth is automatically active as well. This constraint redundancy can be remedied by reparameterizing the model in terms of neighboring pixel differences $\bx = \bD \bu$.  Unfortunately, the rank deficiency of $\bD$ is also an issue.  Adding the same constant to all of the components of $\bu$ yields exactly the same $\bx$.  To circumvent this problem, we simply append a bottom row to $\bD$ with all entries 0 except for a 1 in the last position. If $\bV$ is the amended version of $\bD$, then $\bV$ has full column rank, and the vector $\bx = \bV \bu$ uniquely determines the image.  Indeed, one can solve for $\bx$ in the form $\bu = (\bV^t \bV)^{-1} \bV^t \bx$.  The bottom entry of $\bx$ is obviously the gray level of the last pixel of the image.

Despite the presence of the inverse  of the huge $mn \times mn$ matrix $\bV^t \bV\!$, the transformation $\bu = (\bV^t \bV)^{-1} \bV^t\bx$ is not as daunting as it appears.  First of all, multiplication by the sparse matrix $\bV^t$ is trivial. More importantly, the matrix $\bV^t \bV$ is symmetric, banded, and extremely sparse.  To count its nonzero entries, note that except for diagonal entries, these entries occur in the same positions as the nonzero entries of the adjacency matrix of a corresponding graph with $2mn-m-n$ edges and $mn$ nodes. Because an adjacency matrix has twice as many nonzero entries as edges, the matrix $\bV^t \bV$ has at most $2(2mn-m-n)+mn=5mn-2m-2n$ nonzero entries. These occur within a band of width $\min\{m,n\}$ along the main diagonal, depending on whether we stack columns or concatenate rows.  The most convenient way to solve equations of the kind $\bV^t \bV \ba = \bb$ is to extract the Cholesky decomposition $\bL$ of $\bV^t \bV$ and execute forward and backward substitution.  Although extraction of $\bL$ is cheap for banded matrices, it is even cheaper for banded matrices with just a handful of nonzero entries per row.  In our experience, the computational complexity of extracting $\bL$ scales linearly in the product $mn$. Since $\bL$ itself is sparse, forward and backward substitution are also very cheap. For instance with a $256 \times 256$ image, {\sc Matlab} computes $\bL$ (a $65536 \times 65536$ matrix) in 0.26 seconds on a laptop; $\bL$ contains just 1,971,395 nonzero entries. The sparsity of $\bL$ suggests that it be computed once and stored in compressed format for all images of a given size. Many of its nonzero entries are close to zero. Thus, a fairly light truncation of the non-diagonal entries of $\bL$ gives an even sparser matrix realizing nearly the same transformation. Figure \ref{fig:sparsity-pattern} displays the sparsity pattern of the matrix $\bV^t\bV$ and its permuted Cholesky factor $\bL$ for $64 \times 64$ images. Images of other sizes show similar sparsity patterns.
\begin{figure}
\begin{center}
$$
\begin{array}{cc}
\includegraphics[width=2.25in]{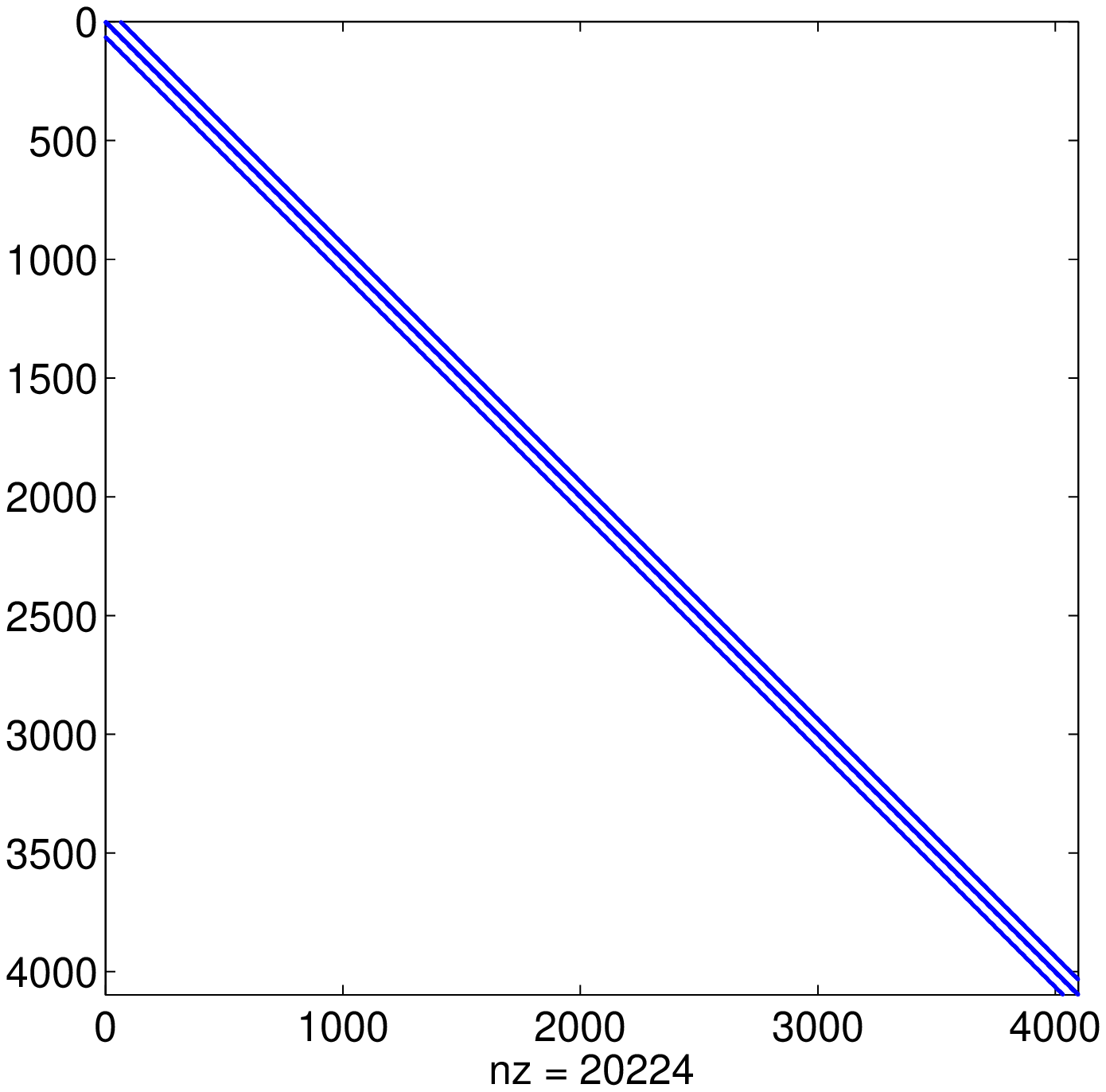} & \includegraphics[width=2.25in]{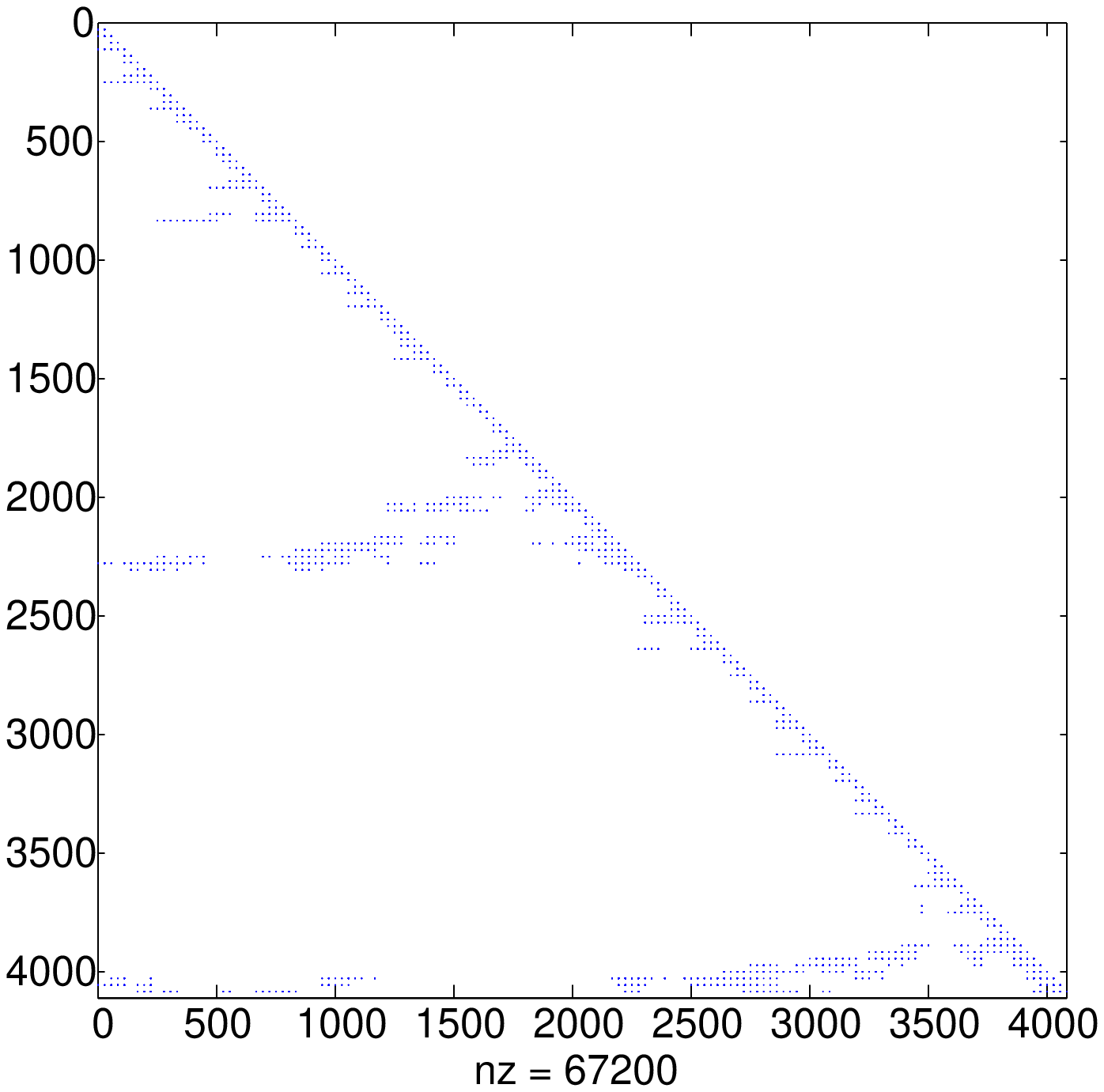}
\end{array}
$$
\end{center}
\caption{Sparsity patterns of $\bV^t\bV$ and its Cholesky decomposition $\bL$ for 64-by-64 images.}
\label{fig:sparsity-pattern}
\end{figure}

The problem of minimizing the objective function $\frac{1}{2} \|\bw - \bK \bu \|_2^2  + \rho \|\bD \bu\|_1$ in the transformed variable
$\bx$ turns out to coincide with lasso penalized regression, for which an efficient path algorithm is known \citep{EfronHastieIainTibshirani04LARS,OsbornePresnellTurlach00LassoAlgo}.  Let us sketch how path following works in the more general case.  The objective function is $f(\bB \bx)+ \rho \|\bx_-\|_1$, where $\bB =  (\bV^t \bV)^{-1} \bV^t$ and $\bx_-$ denotes the vector $\bx$ with its last entry deleted. The penalty contributions correspond to affine equality constraints in constrained minimization. In path following, the penalty constant $\rho$ starts large and moves downward. The initial image is flat with gray level determined by taking $\bx_-={\bf 0}$ and adjusting the last entry of $\bx$ to minimize $f(\bB \bx)$.  Call this point $\bx_\infty$. The first escape time occurs at $\rho_{\text{max}}=\max_j |(\bB^t \nabla f(\bB \bx_\infty)_j|$.  At this juncture path following begins in earnest. Under the $\bx$ parameterization, the loss function has gradient $\bB^t \nabla f(\bB \bx)$ and second differential $\bB^t d^2f(\bB \bx) \bB$.  Because $f(\bB\bx)$ is not strictly convex, our previous reparameterization from $\bx$ to $\by$ variables is needed. Based on equation (\ref{eqn:sol-ode-alt}), the path ODEs reduce to
\begin{eqnarray}
    \frac{d}{d\rho} \bx_{\bar {\mathcal Z}} &=& - (\bB_{\bar {\mathcal Z}}^t d^2f(\bB\bx) \bB_{\bar { \mathcal Z}})^{-1} \text{sgn}(\bx_{\bar {\mathcal Z}}), \hspace{.2in}
    \frac{d}{d\rho} \bx_{\mathcal Z} = {\bf 0},   \nonumber \\
    \frac{d}{d\rho}\blambda_{\mathcal Z} &=& - \bB_{\mathcal Z}^t d^2f(\bB\bx) \bB_{\bar {\mathcal Z}} \frac{d}{d\rho} \bx_{\bar {\mathcal Z}}.
    \label{lasso_directions}
\end{eqnarray}
Observe that the updates of equation (\ref{eqn:sol-ode-alt}) drastically simplify because the rows of the active constraint matrix $\bU_{\mathcal Z}$ and the columns of its null space matrix $\bY$ are populated by standard Euclidean unit vectors. Furthermore, for the ROF model of image denoising, $d^2f(\bB\bx)$ is a diagonal matrix. Alternatively, one can derive the ODE equations (\ref{lasso_directions}) from first principles by implicitly differentiating the stationary conditions. Path following solves the coupled ODEs (\ref{lasso_directions}) segment by segment.

For a quadratic loss function, the second differential is constant, and the solution path is piecewise linear. Thus no ODE solving is involved. With a blurring matrix $\bK$, the second differential is $\bB^t d^2f(\bB \bx) \bB = \bB^t \bK^t \bK \bB$.  After each path extension, the path directions (\ref{lasso_directions}) yield the next event time $\rho_j$ at which a nonzero component $x_j$ hits zero, or a multiplier $\lambda_j$ of a zero component $x_j$ hits $\rho$ or $-\rho$.  The path is then extended to the closest of these event times. In deblurring or denoising, the inverse of $\bB_{\bar {\mathcal Z}}^t \bK^t \bK \bB_{\bar { \mathcal Z}}$ is best computed via a QR decomposition of $\bB_{\bar {\mathcal Z}}\bK$.  At each kink in the path, $\bB_{\bar {\mathcal Z}}\bK $ changes by adding or deleting a column of $\bB \bK$. As we mentioned earlier, it is straightforward to update or downdate the QR decomposition \citep{LawsonHanson87LSBook}. In the original ROF model, traversing one time segment requires about $O(p)$ operations for $p=mn$ total pixels. The whole process ends when $T$ differences $x_j$ becomes nonzero.  In practice, a large value of $T$ recovers too grainy an image, so $T$ is typically much smaller than $p$.   The total cost of computing the solution path is approximately $O(Tp)$,  which is comparable to the cost of start-of-art algorithms for minimization at a single $\rho$.

Figure \ref{fig:lighthouse} illustrates denoising of a $112 \times 91$ image of a lighthouse. The corrupted image appears in the top-left corner of the
figure. The $p=10,192$ pixels generate $20,182$ transformed variables. It takes our {\sc Matlab} script about one minute of desktop computing time  to traverse $T=2,500$ segments along the regularization path from $\rho=87.9881$ (blank image) to $\rho=0.5206$ (a nearly optimal image).  In the
process, the lighthouse clearly emerges from the fog of oversmoothing.  Figure \ref{fig:lighthouse} displays selected snapshots along the regularization path. We emphasize that path following based on equation (\ref{eqn:denoise-gen}) reveals the entire path for the interval  [0.5206,87.9881] of $\rho$ values. In practice, one can accelerate path following by starting from a $\rho$ nearer to the ultimate destination.

\begin{figure}
\begin{center}
\includegraphics[width=4.5in]{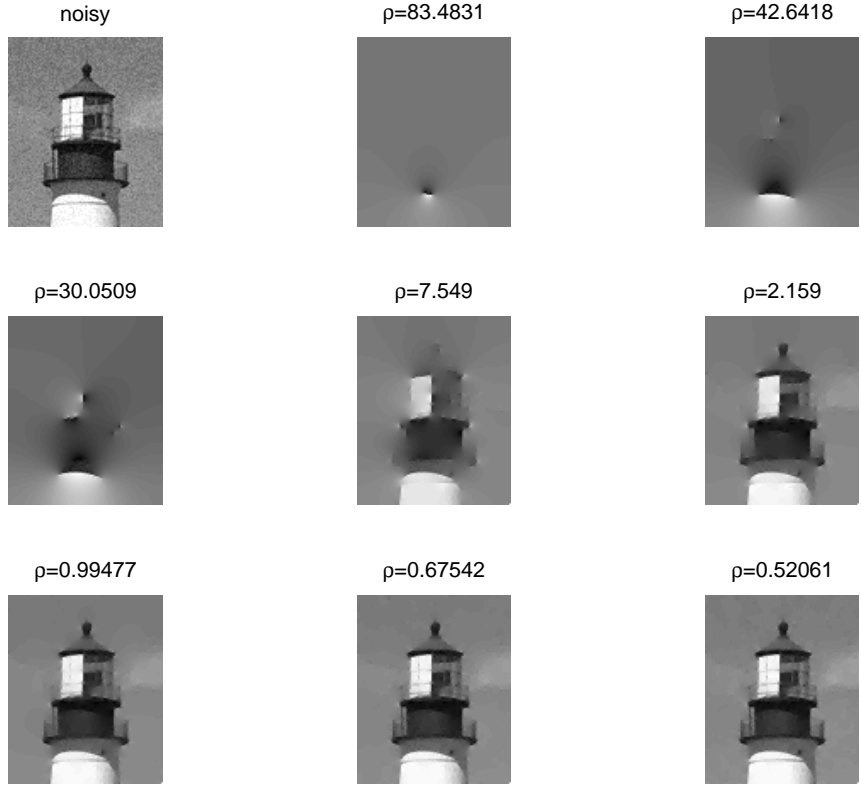}
\end{center}
\caption{A noisy image and snapshots along the regularization path.}
\label{fig:lighthouse}
\end{figure}

\section{Discussion}
\label{sec:conclusions}

Our path following algorithm for constrained convex optimization builds on but differs from the tradition of path following in homotopy methods \citep{zangwill1981pathways} and interior point programming \cite{BoydVandenberghe04Book}.  The paths encountered in the exact penalty method introduce the novelty of piecewise differentiability, which can be effectively handled by tracking the Langrange multipliers.  Computational statisticians deserve credit for exploring this difficult terrain \citep{EfronHastieIainTibshirani04LARS,OsbornePresnellTurlach00LassoAlgo,ZhouLange11LSPath,ZhouWu11EPSODE}. To our knowledge we are the first to make the connection to exact penalty methods.

Our algorithm enjoys the dual advantages of simplicity and generality.  Given the rich numerical resources of {\sc Matlab}, it is straightforward to solve the required ODEs segment by segment. Regardless of whether path following is faster or slower than existing optimization methods, it supplies the whole solution path.  In regularized estimation, this level of detail offers unprecedented insight into how penalties and predictors interact. Our example on image denoising is a case in point.

In quadratic programming with affine equality and inequality constraints, the solution path is piecewise linear \citep{ZhouLange11LSPath}.  This permits path following to take large steps.  Furthermore, each step can be implemented very efficiently by the sweep operator of computational statistics.  Despite the loss of these advantages in more complicated examples, the real culprit in path-following deceleration in many applications is an excessive number of constraints to be navigated.  Our image denoising example suffers from this defect.  On the positive side of the ledger, in nonconvex problems path following may well prove to be more reliable than competing methods in separating global from local minima \citep{ZhouLange09Annealing}.

Various extensions of path following are in order. First, the current algorithm commences from the unconstrained solution.  Our development
relies on the strict convexity and coerciveness of the objective function to ensure a unique starting point.  In principle, path initiation should work for any problem with a unique unconstrained minimum.  Similarly, path continuation should be possible whenever the interior solution is well defined and piecewise smooth. As the image denoising example suggests, reparametrization can play an important role in correcting defects in strict convexity.  Another possibility is to amend the surrogate function ${\mathcal E}_{\rho}(\bx)$. In our semidefinite programming example, we add the term $e^{-c \rho}\|\bX\|_{\text{F}}^2$ to enforce strict convexity and coerciveness.  A similar tactic obviously works in other examples.

A second generalization is to expand the list of penalty functions.  For instance, Euclidean penalties of the form $\|\bM \bx + \ba\|_2$ are useful in grouping parameters in statistical problems.  It should be straightforward to extend path following to include such penalties. A third generalization is to remove convexity restrictions altogether.  As we have noted, the exact penalty method applies equally to nonconvex programming. Path following in this setting is nontrivial since the solution path is no longer necessarily continuous. This poses a real challenge, and it is unclear to us whether one can construct a theory as satisfying as that standing behind modern interior point methods.  We invite the optimization community to tackle this broader issue.  In the meantime, we are happy to share our {\sc Matlab} code with interested researchers.

\bibliography{../../../bib-HZ}
\bibliographystyle{Chicago}

\end{document}